\title{Fast embedding of spanning trees in biased Maker-Breaker games}
\author{{Asaf Ferber \thanks{School of Mathematical Sciences, Raymond and Beverly Sackler Faculty of Exact Sciences, Tel Aviv University, Tel Aviv, 69978, Israel. Email: ferberas@post.tau.ac.il.}} \quad {Dan Hefetz \thanks{Institute of Theoretical Computer Science, ETH Zurich, CH-8092 Switzerland. Email: dan.hefetz@inf.ethz.ch.}} \quad {Michael Krivelevich
\thanks{School of Mathematical Sciences, Raymond and Beverly Sackler Faculty of Exact
Sciences, Tel Aviv University, Tel Aviv, 69978, Israel. Email: krivelev@post.tau.ac.il.
Research supported in part by USA-Israel BSF Grant 2006322 and by grant 1063/08 from the Israel
Science Foundation.}}}

\documentclass[12pt]{article}
\usepackage{amsmath,amssymb,latexsym,color,epsfig,a4}

\newif\ifnotesw\noteswtrue

\def\({\left(}
\def\){\right)}

\newtheorem{theorem}{Theorem}[section]
\newtheorem{lemma}[theorem]{Lemma}
\newtheorem{claim}[theorem]{Claim}
\newtheorem{proposition}[theorem]{Proposition}

\newtheorem{problem}[theorem]{Problem}

\renewcommand{\epsilon}{\varepsilon}

\newenvironment{proof}{\noindent{\bf Proof\,}}{\hfill$\Box$}

\begin{document}
\maketitle

\begin{abstract}
Given a tree $T=(V,E)$ on $n$ vertices, we consider the $(1 : q)$ Maker-Breaker
tree embedding game ${\mathcal T}_n$. The board of this game is the edge set of 
the complete graph on $n$ vertices. Maker wins ${\mathcal T}_n$ if and only if he 
is able to claim all edges of a copy of $T$. We prove that there exist
real numbers $\alpha, \varepsilon > 0$ such that, for sufficiently large $n$ 
and for every tree $T$ on $n$ vertices with maximum degree at most $n^{\varepsilon}$, 
Maker has a winning strategy for the $(1 : q)$ game ${\mathcal T}_n$,
for every $q \leq n^{\alpha}$. Moreover, we prove that
Maker can win this game within $n + o(n)$ moves which is clearly asymptotically optimal.
\end{abstract}

\section{Introduction}
\label{sec::intro}

Let $X$ be a finite set and let ${\mathcal F} \subseteq 2^X$ be a family
of subsets. In the $(p : q)$ Maker-Breaker game $(X,{\mathcal F})$,
two players, called Maker and Breaker, take turns in claiming previously
unclaimed elements of $X$, with Breaker going first. The set $X$ is called
the \emph{board} of the game and the members of ${\mathcal F}$ are referred
to as the \emph{winning sets}. Maker claims $p$ board elements per turn,
whereas Breaker claims $q$. The parameters $p$ and $q$ are called the \emph{bias}
of Maker and of Breaker respectively. Maker wins the game as soon as he occupies all
elements of some winning set. If Maker does not fully occupy any winning set 
by the time every board element is claimed by some player, then Breaker 
wins the game. We say that the $(p : q)$ game $(X,{\mathcal F})$ is 
\emph{Maker's win} if Maker has a strategy that ensures his win in this game 
(in some number of moves) against any strategy of
Breaker, otherwise the game is \emph{Breaker's win}. Note that $p, q,
X$ and ${\mathcal F}$ determine whether the game is Maker's win
or Breaker's win.

Let $T=(V,E)$ be a tree on $n$ vertices. In this paper we study the
biased tree embedding game ${\mathcal T}_n$. The board of ${\mathcal T}_n$
is $E(K_n)$, that is, the edge set of the complete graph on $n$ vertices.
The winning sets of ${\mathcal T}_n$ are the copies of $T$ in $K_n$. 

There are three natural questions which come to one's mind regarding 
the game ${\mathcal T}_n$: 
\begin{description} 
\item [$1.$] For which trees $T$ can Maker win the $(1 : q)$ game ${\mathcal T}_n$?
\item [$2.$] What is the largest positive integer $q$ for which Maker can win the 
$(1 : q)$ game ${\mathcal T}_n$? 
\item [$3.$] How fast can Maker win the $(1 : q)$ game ${\mathcal T}_n$ (assuming it is Maker's win)?
\end{description}

It is easy to see that, playing a $(1 : q)$ game on $E(K_n)$, Breaker can ensure 
that the maximum degree of Maker's graph will not exceed $\frac{n-1}{\lfloor q/2 \rfloor + 1}$. Hence, Maker cannot build any tree on $n$ vertices whose maximum degree exceeds this bound.
 
It is well known (see~\cite{CE}) that if $q \geq \left(1 + \varepsilon \right) \frac{n}{\log n}$, then Breaker can isolate a vertex in Maker's graph. In particular, playing against such Breaker's bias, Maker cannot build any spanning graph. On the other hand, it was proved in~\cite{K} that, if $q \leq \left(1 - \varepsilon \right) \frac{n}{\log n}$, then Maker can build a Hamilton path (in fact, even a Hamilton cycle). Hence, there are spanning trees for 
which $\log n/n$ is the breaking point between Maker's win and Breaker's win. Note that
the requirement that $T$ is a spanning tree plays a crucial role in the aforementioned 
bound on Breaker's bias. Indeed, it was proved by Beck~\cite{Beck94} that, for 
sufficiently large $n$, if $q \leq n/(100d)$, then, playing a $(1 : q)$ game on $E(K_n)$,
Maker can build a $(q,d)$-tree-universal graph, that is, a graph which contains a copy
of \emph{every} tree on $q$ vertices with maximum degree at most $d$.

Clearly, Maker cannot build a spanning tree of $K_n$ in less than $n-1$ moves. For 
certain trees this trivial lower bound is tight. Indeed, it was proved 
in~\cite{HKSS} that, playing a $(1 : 1)$ game on $E(K_n)$, Maker can claim all
edges of a Hamilton path of $K_n$ in $n-1$ moves. Moreover, if Maker just wants to
build a connected spanning graph, that is, he does not have to declare in advance
which spanning tree he intends to build, then he can do so in $n-1$ moves even
in a $(1 : (1 - \varepsilon) n/\log n)$ game (see~\cite{GS}). On the other hand, it was 
conjectured by Beck~\cite{Beck94} and subsequently proved by Bednarska~\cite{Bednarska} 
that, playing a $(1 : q)$ game on $E(K_n)$, where $q \geq c n$ for an arbitrarily small constant $c > 0$, Maker \emph{cannot} build a complete binary tree on $q$ vertices in optimal time, that is, in $q-1$ moves. It seems plausible that, assuming Maker can win the $(1 : q)$ game ${\mathcal T}_n$, he can in fact win it within $n + o(n)$ moves.    

Our main result gives a partial answer to the three aforementioned questions.

\begin{theorem} \label{spanningTree}
Let $0 < \alpha < 0.005$ and $0 < \varepsilon < 0.05$ be real numbers and let $n$ be sufficiently large (that is, $n \geq n_0(\alpha, \varepsilon)$). Let $T=(V,E)$ be a tree
on $n$ vertices, with maximum degree $\Delta(T)\leq n^{\varepsilon}$.
Then, Maker (as the first or second player) can win the $(1 : b)$
Maker-Breaker tree embedding game ${\mathcal T}_n$, for
every $b \leq n^{\alpha}$, in $n+o(n)$ moves.
\end{theorem}

\noindent The rest of this paper is organized as follows: in 
Subsection~\ref{subsec::prelim} we introduce some notation and terminology
that will be used throughout this paper.
In Section~\ref{sec::tools} we state and prove several auxiliary results
which will be used in the proof of Theorem~\ref{spanningTree}. 
In Section~\ref{sec::proofs} we prove Theorem~\ref{spanningTree}.
Finally, in Section~\ref{sec::openprob} we present some open problems.

\subsection{Notation and terminology} \label{subsec::prelim} 
\noindent For the sake of simplicity and clarity of presentation, we
do not make a par\-ti\-cu\-lar effort to optimize the constants
obtained in our proofs. We also omit floor and
ceiling signs whenever these are
not crucial. Most of our results are asymptotic in nature and
whenever necessary we assume that $n$ is sufficiently large.
Throughout the paper, $\log$ stands for the natural logarithm,
unless stated otherwise. Our
graph-theoretic notation is standard and follows that
of~\cite{West}. In particular, we use the following.

For a graph $G$, let $V(G)$ and $E(G)$ denote its sets of vertices
and edges respectively, and let $v(G) = |V(G)|$ and $e(G) = |E(G)|$.
For disjoint sets $A,B \subseteq V(G)$, let $E_G(A,B)$ denote the set of
edges of $G$ with one endpoint in $A$ and one endpoint in $B$, and
let $e_G(A,B) = |E_G(A,B)|$. For a set $S \subseteq V(G)$, let
$N_G(S) = \{u \in V(G) \setminus S: \exists v \in S, (u,v)
\in E(G)\}$ denote the set of neighbors of the vertices of $S$ in $V(G) \setminus S$. 
For a vertex $w \in V(G)$ we abbreviate $N_G(\{w\})$ to $N_G(w)$, and 
let $d_G(w) = |N_G(w)|$ denote the degree of $w$ in $G$. The maximum degree
of a graph $G$ is denoted by $\Delta(G)$. For vertices
$u,v \in V(G)$ let $dist_G(u,v)$ denote the \emph{distance} between $u$
and $v$ in $G$, that is, the number of edges in a shortest path of $G$, 
connecting $u$ and $v$. Often, when there is no risk of confusion, we omit the subscript $G$
from the notation above. For a set $S \subseteq V(G)$, let $G[S]$ denote the 
subgraph of $G$, induced on the vertices of $S$. Let $P = (v_0, \ldots, v_k)$
be a path in a graph $G$. The vertices $v_0$ and $v_k$ are called the \emph{endpoints}
of $P$, whereas the vertices of $V(P) \setminus \{v_0, v_k\}$ are called the
\emph{interior vertices} of $P$. We denote the set of endpoints of a path $P$ by
$End(P)$. A path of a tree $T$ is called \emph{a bare path}
if all of its interior vertices are of degree 2 in $T$.
Given two graphs $G$ and $H$ on the same set of vertices $V$, 
let $G \setminus H$ denote the graph with vertex set $V$ and edge set 
$E(G)\setminus E(H)$. A graph $G=(V,E)$ is said to be \emph{Hamilton connected}
if, for every two vertices $u,w \in V$, there is a Hamilton path in $G$ whose
endpoints are $u$ and $w$. A \emph{triangle factor} of a graph $G$ is a spanning
$2$-regular subgraph of $G$, every connected component of which is isomorphic to $K_3$.  

Let $G$ be a graph, let $T$ be a tree, and let
$S \subseteq V(T)$ be an arbitrary set. An \emph{$S$-partial
embedding} of $T$ in $G$ is an injective mapping $f : S \rightarrow
V(G)$, such that $(f(x),f(y)) \in E(G)$ whenever $\{x,y\} \subseteq
S$ and $(x,y) \in E(T)$. For every vertex $v\in f(S)$ we denote
$v'=f^{-1}(v)$. If $S = V(T)$, we call an $S$-partial embedding of
$T$ in $G$ simply an embedding of $T$ in $G$. We say that the
vertices of $S$ are \emph{embedded}, whereas the vertices of $V(T)
\setminus S$ are called \emph{new}. An embedded vertex is called
\emph{closed} with respect to $T$ if all its neighbors in $T$ are embedded as well. An
embedded vertex that is not closed with respect to $T$, is called
\emph{open} with respect to $T$. The vertices of $f(S)$ are called \emph{taken}, whereas
the vertices of $V(G) \setminus f(S)$ are called \emph{available}.
With some abuse of this terminology, for a closed (respectively
open) vertex $u \in S$, we will sometimes refer to $f(u)$ as being
closed (respectively open) as well. 

Assume that some Maker-Breaker
game, played on the edge set of some graph $G$, is in progress. At
any given moment during this game, we denote the graph spanned by
Maker's edges by $M$, and the graph spanned by Breaker's edges by
$B$. At any point of the game, the edges of $G \setminus (M \cup B)$
are called \emph{free}. We also denote by $d_M(v)$ and $d_B(v)$ the degree
of a given vertex $v \in V$ in $M$ and in $B$ respectively.

\section{Auxiliary results}
\label{sec::tools} 

In this section we present some auxiliary results
that will be used in the proof of Theorem~\ref{spanningTree}. 

The following fundamental theorem, due to Beck~\cite{Beck}, is a useful
sufficient condition for Breaker's win in the $(p:q)$ game $(X,
{\mathcal F})$. It will be used extensively throughout the paper.
\begin{theorem} \label{bwin}
Let $X$ be a finite set and let ${\mathcal F} \subseteq 2^X$. If
$\sum_{B \in {\mathcal F}}(1+q)^{-|B| / p} < \frac{1}{1+q}$, then
Breaker (as the first or second player) has a winning strategy for the
$(p : q)$ game $(X, {\mathcal F})$.
\end{theorem}

While Theorem~\ref{bwin} is useful in proving that Breaker wins a certain game,
it does not show that he wins this game quickly. The following lemma is 
helpful in this respect.

\begin{lemma} [Trick of fake moves] \label{lem::fakeMoves}
Let $X$ be a finite set and let ${\mathcal F} \subseteq 2^X$. Let $q' < q$ be 
positive integers. If Maker has a winning strategy for the $(1 : q)$ game 
$(X, {\mathcal F})$, then he has a strategy to win the $(1 : q')$ game 
$(X, {\mathcal F})$ within $1 + |X|/(q+1)$ moves.
\end{lemma}

The main idea of the proof of Lemma~\ref{lem::fakeMoves} is that, in every move
of the $(1 : q')$ game $(X, {\mathcal F})$, Maker (in his mind) gives Breaker
$q - q'$ additional board elements. The straightforward details can be found in~\cite{BeckBook}. 

Let $T=(V,E)$ be an arbitrary tree on $n$ vertices. For every $1 \leq i
\leq n-1$ let $D_i := \{v \in V : d_T(v)=i\}$ denote the set of vertices
of $V$ whose degree in $T$ is exactly $i$. Moreover, let
$D_{>i} := \bigcup_{k=i+1}^{n-1} D_k$ denote the set of vertices
of $V$ whose degree in $T$ is strictly larger than $i$.  

\begin{lemma} \label{leaves}
Let $T$ be a tree on $n \geq 2$ vertices, then $|D_{>2}| \leq |D_1| - 2$.
\end{lemma}

\begin{proof}
\begin{eqnarray*}
2n-2 &=& \sum_{v \in V} d(v)\\
&=& \sum_{v \in D_1} d(v) + \sum_{v \in D_2} d(v) + \sum_{v \in D_{>2}} d(v)\\
&\geq& |D_1| + 2(n - |D_1| - |D_{>2}|) + 3|D_{>2}|.
\end{eqnarray*}

It follows that $2n-2 \geq 2n - |D_1| + |D_{>2}|$, and thus 
$|D_{>2}| \leq |D_1| - 2$ as claimed.
\end{proof}

\begin{lemma} \label{randomPartition}
Let $k$ be a sufficiently large positive integer and let $G=(V,E)$ be a 
graph on $k$ vertices with maximum degree at most $k^{0.95}$.
Let $\ell \geq 1$ be an integer and let $L := \{a_1, \ldots,
a_{\ell}, b_1, \ldots, b_{\ell}\}$ be a set of $2\ell$ vertices of
$G$. For every $1 \leq i \leq \ell$, let $k_i$ be an integer such
that $\sum_{i=1}^{\ell} k_i = k - 2\ell$, and $k_i\geq k^{0.2} $.
Then, there exists a partition $V \setminus L = V_1 \cup \ldots \cup
V_{\ell}$ such that the following two properties hold for every $1
\leq i \leq \ell$:
\begin{description}
\item [$(i)$] $|V_i| = k_i$.
\item [$(ii)$] The maximum degree of the graph $G_i := G[V_i \cup \{a_i, b_i\}]$ is at most $10 k_i k^{-0.05}$.
\end{description}
\end{lemma}

\begin{proof}
Let $V \setminus L = V_1 \cup \ldots \cup V_{\ell}$ be a partition,
chosen uniformly at random amongst all partitions of $V \setminus L$
into $\ell$ parts such that $|V_i| = k_i$ for every $1 \leq i \leq
\ell$. Fix some $1 \leq i \leq \ell$, and set $m_i := 10 k_i k^{-0.05}$.
Let $u \in V$ be an arbitrary vertex. The probability that $u$ 
has more than $m_i$ neighbors in $V_i$ is at most $e^{-m_i}$ 
(see Theorem 2.10 and Corollary 2.4 in~\cite{JLR}). 
It follows by a union bound argument, that the probability that
there exists a vertex $u \in V$ such that $d_{G}(u) \geq m_i$ is at
most $k/e^{m_i}$. It thus follows by another union bound argument
that the probability that there exists an index $1 \leq i \leq \ell$
such that $G_i$ does not satisfy property $(ii)$ above, is at most
$\sum_{i=1}^{\ell} \frac{k}{e^{m_i}} \leq
\frac{k \ell}{e^{k^{0.15}}} = o(1)$. In particular, there exists a
partition that satisfies both properties of the lemma.
\end{proof}

\subsection{Playing several biased games in parallel}

Let $m$ be a positive integer. For every $1 \leq i \leq m$, 
let ${\mathcal H}_i = (V_i, E_i)$ be a hypergraph, 
where $V_i \cap V_j = \emptyset$ for every $1 \leq i < j
\leq m$. Let ${\mathcal H} = (V,E)$ be the hypergraph with $V =
\bigcup_{i=1}^m V_i$ and $E = \prod_{i=1}^m E_i=\{\bigcup_{i=1}^m e_i :
e_i \in E_i\}$. Consider a $(1:q)$ Maker-Breaker game played on
${\mathcal H}$. If $q=1$, then Maker can play all $m$ games in
parallel, that is, whenever Breaker claims a vertex of the board 
${\mathcal H}_i$, Maker responds by also claiming a vertex of ${\mathcal H}_i$ according to a
fixed winning strategy for the game ${\mathcal H}_i$ (if Breaker 
claims the last vertex of ${\mathcal H}_i$, then Maker responds by claiming
an arbitrary free vertex of ${\mathcal H}$). It follows that Maker 
wins the $(1 : 1)$ game ${\mathcal H}$ if
and only if he wins the $(1 : 1)$ game ${\mathcal H}_i$ for every
$1 \leq i \leq m$. If $q > 1$, then this is no 
longer true because Breaker can play in $q$ different boards in one turn whereas 
Maker can only respond in one board per turn.
Nevertheless, we prove the following result.
\begin{theorem} \label{ParallelGames} 
If, for every $1 \leq i \leq m$, Maker has a strategy to win the\\ $\left(1 : q \left(1 + \log \left(m + \left\lceil \frac{\sum_{i=1}^m |V_i|}{q+1} \right\rceil \right) \right) \right)$
game ${\mathcal H}_i$ in $t_i$ moves, then he has a strategy to win the $(1 : q)$ game ${\mathcal H}$ in $\sum_{i=1}^m t_i$ moves.
\end{theorem}

Before proving Theorem~\ref{ParallelGames}, we introduce an
auxiliary game, which is a variation on the classical \emph{Box Game},
first introduced by Chv\'atal and Erd\H{o}s~\cite{CE} (see~\cite{HKSSbox} for
a recent development). The {\em Box Game with resets $rBox(m,q)$} is played by two
players, called BoxMaker and BoxBreaker. They play on a hypergraph
${\mathcal H} = \{A_1, \ldots, A_m\}$, where the sets $A_i$ are pairwise disjoint.
BoxMaker claims $q$ elements of $\bigcup_{i=1}^m A_i$ per turn, and then BoxBreaker 
responds by {\em resetting} one of BoxMaker's \emph{boxes}, that is, by deleting all of BoxMaker's elements from the chosen hyperedge $A_i$. Note that the chosen box
does {\em not} leave the game. At every point during the game, 
and for every $1 \leq i \leq m$, we define \emph{the weight} of box 
$A_i$ to be the number of BoxMaker's elements that are currently in $A_i$,
that is, the number of elements of $A_i$ that were claimed by BoxMaker and 
were not yet deleted by BoxBreaker.

\begin{theorem} \label{box1} 
For every integer $k \geq 1$, BoxBreaker has a strategy for the game $rBox(m,q)$ 
which ensures that, at any point during the first $k$ rounds of the game, 
every box $A_i$ has weight at most $q(1 + \log(m+k))$.
\end{theorem}

In order to prove Theorem~\ref{box1}, we introduce a continuous
version of the game $rBox(m,q)$ which we denote by $rCBox(m)$. The board of this game
consists of $m$ boxes $A_1, \ldots, A_m$, of initial weight zero each. 
In each of his moves, CBoxMaker distributes a total weight of 1 among the 
boxes $A_1, \ldots, A_m$ as he pleases, that is, if, for $1 \leq i \leq m$, the current 
weight of $A_i$ is $w_i$, then he changes this weight to $w_i' := w_i + \delta_i$,
where $(\delta_1, \ldots, \delta_m) \in \mathbb{R}^{+}$ is any 
vector such that $\sum_{i=1}^m \delta_i = 1$. CBoxBreaker then resets a box of his choice.
Note that if CBoxBreaker, in the $rCBox(m)$ game, has a strategy to ensure that during the first $k$ rounds of the game every box has weight at most $f(k)$, then BoxBreaker clearly has a strategy in the $rBox(m,q)$ game to ensure a weight of at most $q \cdot f(k)$ in each box. Indeed, whenever BoxMaker claims $q_i \leq q$ elements of a box $A_i$, BoxBreaker
responds according to CBoxBreaker's strategy in $rCBox(m)$, as if BoxMaker has distributed a weight of $q_i/q$ in the box $A_i$. We conclude that Theorem~\ref{box1} is an immediate consequence of the following theorem.

\begin{theorem} \label{box2} For every integer $k \geq 1$, CBoxBreaker has a strategy
to ensure that during the first $k$ rounds of the game $rCBox(m)$
every box $A_i$ has weight at most $1 + \log(m+k)$.
\end{theorem}

\begin{proof}
CBoxBreaker's strategy is fairly straightforward -- he always chooses to reset a
box of maximum weight, breaking ties arbitrarily. 

Define $\phi(x) = e^x$ and observe that for every $\delta > 0$ we have
\begin{equation}\label{lb1}
\phi(x + \delta) - \phi(x) \leq \delta \phi(x + \delta)\ .
\end{equation}

Indeed, since $\phi$ is differentiable, we can apply the Mean Value Theorem to conclude that $\phi(x + \delta) - \phi(x) = \delta \phi'(\theta) = \delta e^{\theta}$ for some $x < \theta < x + \delta$. By the monotonicity of $\phi(x)$ it then follows that $\delta e^{\theta} \leq \delta e^{x + \delta}$.

Now, given a vector of weights $\textbf{w} = (w_1, \ldots, w_m)$, define the potential function $\Phi(\textbf{w})$ by
$$
\Phi(\textbf{w}) = \sum_{i=1}^m \phi(w_i)\ .
$$

Suppose that just before CBoxMaker's $j$th move (where $j \geq 1$ is arbitrary),
for every $1 \leq i \leq m$, the weight of box $A_i$ is $x_i$. 
The potential $\Phi$ before CBoxMaker's $j$th move is thus
$$
\Phi=\sum_{i=1}^m \phi(x_i)=\sum_{i=1}^m e^{x_i}\ .
$$
In his $j$th move CBoxMaker updates the weight of box $A_i$ to $x_i + \delta_i$,
for every $1 \leq i \leq m$. Denote the potential after CBoxMaker's $j$th move 
by $\Phi'$. Then
$$
\Phi' = \sum_{i=1}^m \phi(x_i + \delta_i) = \sum_{i=1}^m e^{x_i + \delta_i}\ .
$$

It follows that
\begin{eqnarray}
\Phi' - \Phi &=& \sum_{i=1}^m e^{x_i + \delta_i} - \sum_{i=1}^m e^{x_i} \nonumber\\
&=& \sum_{i=1}^m \left(e^{x_i + \delta_i} -e^{x_i}\right) \nonumber\\
&\leq& \sum_{i=1}^m \delta_i e^{x_i + \delta_i} \nonumber\\
&\leq& \sum_{i=1}^m \delta_i \cdot \max_i e^{x_i + \delta_i} \nonumber\\
&=& \max_i e^{x_i + \delta_i}, \label{lb2}
\end{eqnarray}

where the first inequality above follows from~\eqref{lb1}.

In his $j$th move, CBoxBreaker resets a box $A_i$ for which
$x_i + \delta_i$ is maximal (breaking ties arbitrarily).
Hence, CBoxBreaker's $j$th move changes the contribution of $A_i$
to the potential $\Phi'$ from $e^{x_i + \delta_i}$ to $e^0 = 1$. 
Denote the potential after CBoxBreaker's $j$th move 
by $\Phi''$. Then
\begin{equation}\label{lb3}
\Phi' - \Phi'' = \exp\{\max_i (x_i + \delta_i)\} - 1 = \max_i
e^{x_i + \delta_i} - 1 \ .
\end{equation}

Combining~\eqref{lb2} and~\eqref{lb3}, we conclude that 
$\Phi'' - \Phi \leq 1$. Therefore, if CBoxBreaker
follows his strategy, then after the first $k$ rounds, the value of
the potential function $\Phi$ increases in total by at most $k$.

Since the potential before CBoxMaker's first move is $m e^0 = m$, the
potential after each of the first $k$ moves of CBoxMaker is at most $m+k$.
It follows that, for every $1 \leq j \leq k$, just before CBoxMaker's $j$th move, 
none of the boxes has weight larger than $\log(m+k)$. In his $k$th move, CBoxMaker
adds a weight of at most 1 to any given box, and thus its weight does not
exceed $1 + \log(m+k)$.
\end{proof}

\textbf{Proof of Theorem~\ref{ParallelGames}}
Let $k = \left\lceil \frac{\sum_{i=1}^m |V_i|}{q+1} \right\rceil$. For every $1 \leq i \leq m$, let ${\mathcal S}_i$ be a strategy for Maker in the $(1 : q(1 + \log(m+k)))$ game ${\mathcal H}_i$ which ensures his win in at most $t_i$ moves. Since the game ${\mathcal H}$ clearly lasts at most $k$ rounds, it follows by Theorem~\ref{box1}, that Maker (assuming the role of BoxBreaker in $rBox(m,q)$) has a strategy ${\mathcal S}$ to ensure that, for every $1 \leq i \leq m$, for every $j \geq 0$, and at any point during the game, if Maker has claimed exactly $j$ vertices of $V_i$, then Breaker has claimed at most $(j+1) q(1 + \log(m+k))$ vertices of $V_i$. For every $i \geq 1$, in his $i$th move in the game ${\mathcal H}$, Maker will use ${\mathcal S}$ for choosing a board ${\mathcal H}_j$ in which to play in this move. If Maker has already won ${\mathcal H}_j$, then he chooses an arbitrary $1 \leq r \leq m$ for which he has not yet won ${\mathcal H}_r$ and plays his $i$th move there. Since Maker chooses a board according to ${\mathcal S}$, it follows by Theorem~\ref{box1} that Breaker has not claimed more than $q(1 + \log(m+k))$ vertices of $V_j$ since Maker has last played on this board. Hence, Maker can follow ${\mathcal S}_j$ whenever he plays in ${\mathcal H}_j$ and thus win this game by assumption. Since this holds for every $1 \leq j \leq m$, it follows that Maker has a winning strategy for the $(1 : q)$ game ${\mathcal H}$. Moreover, since whenever Maker plays in ${\mathcal H}_i$ he follows ${\mathcal S}_i$ and since he never plays in ${\mathcal H}_i$ if he had already won this game, it follows that, for every $1 \leq i \leq m$, Maker plays at most $t_i$ moves in ${\mathcal H}_i$. Hence, he has a strategy to win the $(1 : q)$ game ${\mathcal H}$ within at most $\sum_{i=1}^m t_i$ moves, as claimed.  
{\hfill $\Box$ \medskip\\}

\subsection{A perfect matching game} \label{subsec::matching}
Maker's strategy for embedding a spanning tree, which we will propose in Section~\ref{sec::proofs}, will involve building a perfect matching on some part of the board. Hence, we prove the following result.

\begin{proposition} \label{matchingE}
Let $r$ be a sufficiently large integer and let $q \leq 
\frac{r}{12\log_2 r}$. Let $G$ be a spanning subgraph of $K_{r,r}$ with minimum
degree at least $r - g(r)$, where $g$ is an arbitrary function satisfying
$g(r) = o(r)$. Then, playing a $(1 : q)$ game on $E(G)$,
Maker can claim the edges of a perfect matching of $G$,
within $O(r \log r)$ moves.
\end{proposition}

\begin{proof}
Let $A$ and $B$ denote the two partite sets of $G$.
In order to show that Maker can claim the edges of a perfect matching 
of $G$, we will prove that Maker can build a graph which satisfies
Hall's condition, that is, a graph $M$ which satisfies
$|N_M(X)| \geq |X|$ for every $X \subseteq A$ (see e.g.~\cite{West}).

We define an auxiliary game ${\mathcal M}_G$, which we refer to as the
\emph{Hall game}. It is a $(q : 1)$ game, played by two players, called
HallMaker and HallBreaker. The board of this game is $E(G)$ and the winning sets
are the edge sets of all induced subgraphs of $G$ with one partite set of
size $1 \leq t \leq r$ and the other of size $r-t+1$. It is straightforward 
to verify that if HallBreaker has a winning strategy for the $(q : 1)$ game
${\mathcal M}_G$, then, playing a $(1 : q)$ game on $E(G)$, Maker can claim 
all edges of some perfect matching of $G$. In order to prove that HallBreaker
can win the $(q : 1)$ game ${\mathcal M}_G$ for $q \leq \frac{r}{12\log_2 r}$, 
we apply Theorem~\ref{bwin}. We have   

\begin{eqnarray*}
\sum_{D \in {\mathcal M}_G} 2^{-\frac{|D|}{q}}
&\leq& \sum_{t=1}^{r} \binom{r}{t} \binom{r}{r-t+1} 2^{-\frac{t(r-t+1) - g(r) \cdot \min\{t, r-t+1\}}{q}}\\
&\leq& 2\sum_{t=1}^{r/2} \binom{r}{t} \binom{r}{t-1} 2^{-\frac{t(r-t+1) - g(r) \cdot t}{q}}\\
&\leq& 2\sum_{t=1}^{r/2} \binom{r}{t}^2 2^{-\frac{t r}{3q}}\\
&\leq& \sum_{t=1}^{r/2} \left[r^2 2^{- 4 \log_2 r} \right]^t\\ 
&\leq& \sum_{t=1}^{r/2} r^{-2t}\\
&=& o(1).
\end{eqnarray*}
It follows that Maker can indeed build the required perfect matching. Moreover, it follows by Lemma~\ref{lem::fakeMoves} that Maker can do so within at most $1 + \frac{|E(G)|}{r/(12\log_2 r)} = O(r \log r)$ moves. This concludes the proof of the proposition.
\end{proof}

\subsection{A Hamiltonicity game} \label{subsec::HamiltonCon}
Maker's strategy for embedding a spanning tree, which we will propose in Section~\ref{sec::proofs}, will involve building a Hamilton connected subgraph of some part of the board. Hence, we prove the following result.

\begin{proposition} \label{HamiltonConnected}
Let $k$ be a sufficiently large integer and let $q \leq \frac{k}{\log^2 k}$.
Let $G=(V,E)$ be a graph on $k$ vertices, with minimum degree at
least $k - g(k)$, where $g$ is an arbitrary function satisfying
$g(k) = o(k/\log k)$. Then, playing a $(1 : q)$ game on $E$, Maker
can build a Hamilton connected graph within $O(k \log^2 k)$
moves.
\end{proposition}

\begin{proof}
In the proof of this proposition we will make use of the following
sufficient condition for a graph to be Hamilton connected (see~\cite{HKS}).

\begin{theorem} \label{HamCon}
Let $D(k)=\log \log k$ and let $G=(V,E)$ be a graph on $k$ vertices
satisfying the following two properties:
\begin{itemize}
\item For every $S \subseteq V$, if $|S| \leq \frac{k}
{\log k}$, then $|N_G(S)| \geq D|S|$;
\item There is an edge in $G$ between any two disjoint
subsets $A,B \subseteq V$ with $|A|,|B| \geq \frac{k}{\log k}$.
\end{itemize}
Then $G$ is Hamilton connected, for sufficiently large $k$.
\end{theorem}
Let ${\mathcal H}_1$ be the hypergraph whose vertices are the edges
of $G$ and whose set of hyperedges is $\{E_G(A,B) : A,B \subseteq V,
A \cap B = \emptyset, 1 \leq |A| \leq \frac{k}{\log k}, |B| = k -
(D+1)|A|\}$. Note that by our assumption on the minimum degree in
$G$, it follows that
\begin{eqnarray*}
e_G(A,B) \geq |A|(|B|-g(k)) \geq (1-o(1))|A|k.
\end{eqnarray*}
for every $A,B$ as above.

Let ${\mathcal H}_2$ be the hypergraph whose vertices are the edges
of $G$ and whose set of hyperedges is $\{E_G(A,B) : A,B \subseteq V,
A \cap B = \emptyset, |A| = |B| = \frac{k}{\log k}\}$. Note that by
our assumption on the minimum degree in $G$, it follows that
\begin{eqnarray*}
e_G(A,B) &\geq& \frac{k}{\log k} \left(\frac{k}{\log k } - g(k)\right)\\
&\geq& \left(\frac{k}{2\log k}\right)^2
\end{eqnarray*}
for every $A,B$ as above.

By Theorem~\ref{HamCon}, in order to prove that playing a $(1 : q)$
game Maker can build a Hamilton connected subgraph of $G$, it
suffices to prove that Breaker can win the $(q : 1)$ game ${\mathcal
H}_1 \cup {\mathcal H}_2$. This however follows from
Theorem~\ref{bwin}. Indeed, for every $1 \leq a \leq \frac{k}{\log k
}$, we have
\begin{eqnarray*}
&& \binom{k}{a} \binom{k}{k - (D+1)a} 2^{-(1-o(1))ak /q}\\
&\leq& k^a k^{(D+1)a} 2^{-(1-o(1))ak /q}\\
&\leq& \exp \left\{a(D+2)\log k - \log 2 \cdot (1-o(1))ak \cdot \frac{\log^2 k}{k}\right\}\\
&=& o(1/k).
\end{eqnarray*}
Hence,
\begin{eqnarray*}
\sum_{B \in {\mathcal H}_1} 2^{-|B| / q} \leq \sum_{a=1}^{\frac{k}{
\log k}} \binom{k}{a} \binom{k}{k - (D+1)a} 2^{-(1-o(1))ak /q} =
o(1).
\end{eqnarray*}

Similarly,
\begin{eqnarray*}
&&\sum_{B \in {\mathcal H}_2} 2^{-|B| / q}\\
&\leq& \binom{k}{\frac{k}{\log k}}^2 \exp \left\{- \log 2 \left(\frac{k}{2\log k}\right)^2 \frac{\log^2 k}{k}\right\}\\
&\leq& \exp \left\{\frac{(2 + o(1)) k\log \log k}{\log k} - \log 2 \left(\frac{k}{2 \log k}\right)^2 \frac{\log^2 k}{k} \right\}\\
&=& o(1).
\end{eqnarray*}

Hence,
\begin{eqnarray*}
\sum_{B \in {\mathcal H}_1 \cup {\mathcal H}_2} 2^{-|B| / q} =
\sum_{B \in {\mathcal H}_1} 2^{-|B| / q} + \sum_{B \in {\mathcal
H}_2} 2^{-|B| / q} = o(1).
\end{eqnarray*}

It follows that Maker can build the required Hamilton connected graph. Moreover, it follows by Lemma~\ref{lem::fakeMoves} that Maker can do so within at most $1 + \frac{|E|}{k/(\log^2 k)} = O(k \log^2 k)$ moves. This concludes the proof of the proposition.
\end{proof}

\section{Embedding a spanning tree quickly}
\label{sec::proofs}

In this section, we prove Theorem~\ref{spanningTree}. We begin by describing Maker's strategy, then prove that it is indeed a winning strategy and that Maker can follow all of its stages.

\textbf{Maker's strategy:} Maker distinguishes between two cases, according to the number of neighbors of the leaves of $T$. Throughout this section, let $L$ denote the set of leaves of $T$.

\textbf{Case I:} $|N_T(L)| \geq n^{2/3}$.

Maker's strategy for this case is divided into two stages.

\textbf{Stage 1:} Let $L' \subseteq L$ be a set of exactly $n^{2/3}$ leaves, every two of which have no common neighbor, that is, $|N_T(L')| = |L'|$.
Maker's goal in this stage is to embed a subtree $T''$ of $T$ such
that $T' := T \setminus L' \subseteq T''$ and $|V(T'')| \leq n - \frac{1}{2}n^{2/3}$, 
in at most $n+o(n)$ moves.

At any point during this stage, a vertex $v \in V(K_n)$ is called \emph{dangerous} if
$d_B(v) \geq \sqrt{n}$ and $v$ is either an available or an open vertex with respect to $T$.
Throughout this stage, Maker maintains a set $D \subseteq V(K_n)$
of dangerous vertices, a set $S \subseteq V(T)$ of embedded
vertices, and an $S$-partial embedding $f$ of $T$ in $K_n \setminus B$. 
Initially, $D = \emptyset$, $S = \{w'\}$, where $w'$ is an arbitrary vertex of $T'$, and
$f(w') = w$, where $w$ is an arbitrary vertex of $K_n$. If at some
point Maker is unable to follow the proposed strategy (including the time
limits it sets), then he forfeits the game. Moreover, if after claiming $2n$ edges,
Maker has not yet won, then he forfeits the game (as noted above, we will in fact
prove that Maker can win within $n + o(n)$ moves; however, the technical upper bound 
of $2n$ will suffice for the time being). 

For as long as $V(T') \setminus S \neq \emptyset$, Maker plays as follows:
\begin{description}
\item [$(1)$] If $D \neq \emptyset$, then Maker plays as follows. Let $v \in D$ be an arbitrary dangerous vertex. We distinguish between two subcases: 
\begin{description}
\item [$(1.1)$] $v$ is taken. Let $v'_1, \ldots, v'_r$ be the new neighbors of $v' := f^{-1}(v)$ with respect to $T$. Maker sequentially claims $r$ free edges $\{(v, v_i) : 1 \leq i \leq r\}$, where, for every $1 \leq i \leq r$, $v_i$ is an arbitrary available vertex. 
Subsequently, he updates $D,f$ and $S$ by adding $v'_i$ to $S$ and
setting $f(v'_i) = v_i$ for every $1 \leq i \leq r$, and by adding (if necessary)
new dangerous vertices to $D$. 

\item [$(1.2)$] $v \in D$ is available.  
First, Maker adds $v$ to his tree. If there
exists an open vertex $u$ such that $(u,v)$ is a free edge, then Maker
claims this edge, and then closes $v$ and updates $D,f$ and $S$ as in $(1.1)$.
If there is no such edge, that is, $v$ is connected by Breaker's edges
to every open (with respect to $T$) vertex of Maker's tree,
then Maker connects $v$ to an open vertex $u$ via a path of
length three, within at most $11 n^{\alpha}$ moves. Finally, Maker
closes $v$ and updates $D,f$ and $S$ as in $(1.1)$. 
\end{description}

\item [$(2)$] If $D = \emptyset$, then Maker claims some free
edge $(u,v)$ such that $v$ is an open vertex with respect to $T'$ and 
$u$ is an available vertex. 

\end{description}

\textbf{Stage 2:} Maker embeds the vertices of $T \setminus T''$ 
(that is, all vertices of $T$ that were not embedded in Stage 1) into
the set of available vertices, within $o(n)$ moves.

\vspace{7.5mm}

\textbf{Case II:} $|N_T(L)| < n^{2/3}$. 
It follows that $T$ has strictly less than $n^{2/3 + \varepsilon}$ leaves. 

Maker's strategy for this case is divided into two stages.

\textbf{Stage 1:} Let $F$ denote the forest which is obtained from $T$ by removing the interior vertices of all inclusion maximal bare paths whose length is at least $n^{0.2}$. Maker embeds $F$ in $K_n$ without paying any attention to Breaker's moves.  

\textbf{Stage 2:} Maker embeds the edges of $T \setminus F$, thus completing the embedding of $T$. In order to do so, he splits the rest of the board into parts of appropriate sizes and, playing on all parts in parallel, he embeds the missing pieces of $T$ in the appropriate parts.

Note that if Maker can indeed follow all parts of the proposed strategy, then he clearly
wins the $(1 : q)$ game ${\mathcal T}_n$ (though possibly not fast enough).

\subsection{Following Maker's strategy for Case I}
In this subsection we prove that Maker can indeed follow every part of
his strategy for Case I. We prove this separately for Stage 1 and for Stage 2.
First, we prove the following two lemmas.

\begin{lemma} \label{dang}
At any point during Stage 1 there are $O(n^{1/2 + \alpha}) = o(n^{0.6})$
dangerous vertices.
\end{lemma}

\begin{proof}
By Maker's strategy, the game lasts at most $2n$ moves. Since, moreover, every
dangerous vertex has degree at least $\sqrt{n}$ in Breaker's graph, it follows that
there can be at most $2 \cdot \frac{2 n^{1 + \alpha}}{\sqrt{n}} = O(n^{1/2 + \alpha})$ 
such vertices.
\end{proof}

\begin{lemma} \label{maintainProperties}
The following two properties hold at any point during Stage 1:
\begin{itemize}
\item [$(i)$] $n - |S| \geq \frac{1}{2} n^{2/3}$.
\item [$(ii)$] $d_B(v) = o(n^{0.6})$ holds for every $v \in V(K_n) \setminus f(S)$
and for every $v \in f(S)$ which is open with respect to $T$.
\end{itemize}
\end{lemma}

\begin{proof}
\begin{description}
\item [$(i)$] By Maker's strategy, for every dangerous vertex $v \in D$, Maker embeds less than $2n^{\varepsilon}$ other vertices ($n^{\varepsilon}$ vertices for closing $v$ and $3$ more if he is forced to first add $v$ to his tree via a path of length $3$ as described in part $(1.2)$ of Maker's strategy). Other than these vertices, Maker embeds only vertices of $T'$. It follows by Lemma~\ref{dang} that $|S| \leq |V(T')| + o(n^{0.6 + \varepsilon}) \leq (n - n^{2/3}) + o(n^{0.6 + \varepsilon}) \leq n - \frac{1}{2} n^{2/3}$ holds at any point during Stage 1.

\item [$(ii)$] This holds trivially for every non-dangerous vertex. By Maker's strategy, closing any dangerous vertex, requires at most $n^{\varepsilon} + 11 n^{\alpha}$ moves. It follows by Lemma~\ref{dang} there are at most $O(n^{1/2 + \alpha})$ dangerous vertices at any point during Stage 1. Moreover, as long as $D \neq \emptyset$, all of Maker's moves are dedicated to closing dangerous vertices and, after each move of Maker, Breaker claims just $n^{\alpha}$ free edges of $K_n$. Hence, when Maker tries to close a dangerous vertex $v$, it holds that $d_B(v) \leq \sqrt{n} + O(n^{\max\{\varepsilon, \alpha\}} n^{\alpha} n^{1/2 + \alpha}) = o(n^{0.6})$. Since, unless he has already won, Maker continues closing dangerous vertices for as long as $D \neq \emptyset$, it follows that $d_B(w) = o(n^{0.6})$ holds for every $w \in V(K_n) \setminus f(S)$ and for every $w \in f(S)$ which is open with respect to $T$, as claimed.
\end{description}
\end{proof}

\textbf{A proof that Maker can follow Stage 1 of his strategy for Case I:} 

For as long as $D = \emptyset$ (part $(2)$ of Maker's strategy), Maker claims a free edge
$(v,u)$, where $v' = f^{-1}(v)$ is open with 
respect to $T'$ and $u \in V(K_n)$ is available. 
This is always possible since, unless Maker has already won, there 
must exist an open vertex $v'$ with respect to $T'$. Moreover,
since $D = \emptyset$, it follows that $d_B(v) < \sqrt{n}$. However,
it follows by part $(i)$ of Lemma~\ref{maintainProperties} that 
$|V(K_n) \setminus f(S)| \geq \frac{1}{2} n^{2/3}$. Hence, there exists an available vertex 
$u$ for which $(v,u)$ is free.

It remains to consider part $(1)$ of Maker's strategy, that is, the case $D \neq \emptyset$.

\emph{Part $(1.1)$ of Maker's strategy -- $v$ is already taken.}
It follows that $v' = f^{-1}(v)$ is open with respect to $T$. 
Maker's goal is to close $v$ in $T$, 
that is, to embed all vertices of $N_T(v') \setminus S$ into available vertices of
$V(K_n)$. We claim that the required edges exist. Indeed, it follows by
part $(ii)$ of Lemma~\ref{maintainProperties} that $d_B(v) = o(n^{0.6})$ 
holds for as long as $v$ is open. Since $|V(K_n) \setminus f(S)| \geq \frac{1}{2} n^{2/3} > n^{0.6}$ holds by part $(i)$ of Lemma~\ref{maintainProperties}, it follows that the
required available vertices and corresponding free edges exist.

\emph{Part $(1.2)$ of Maker's strategy -- $v$ is an available vertex.} In this case, Maker must 
first add $v$ to his current tree. As noted in his strategy, if there exists 
an open vertex $u$ such that the edge $(u,v)$ is free, then Maker claims this edge.
Otherwise, Maker connects $v$ to an open (with respect to $T$) vertex of his 
current tree via a path of length three. First, we claim that there exists a taken vertex $u$, such that $u' := f^{-1}(u)$ is open with respect to $T$, and new vertices $v', x', y' \in V(T)$ such that $(u', x'), (x', y'), (y', v') \in E(T)$. Indeed, assume for the sake of contradiction that this is not the case. It follows from our assumption that, for every new vertex $a' \in V(T)$, there exists an open vertex $b' \in V(T)$ such that $dist_T(a',b') \leq 2$. Hence, every new vertex is an element of $L \cup N_T(L)$. It follows by the definition of $L'$ and by part $(i)$ of Lemma~\ref{maintainProperties} that at least $\frac{1}{2} n^{2/3}$ of the vertices of $N_T(L)$ are either open or new (in which case it must have an open neighbor). Hence, the number of open vertices in Maker's tree with respect to $T$ is at least $\frac{n^{2/3}/2}{\Delta(T)} \geq \frac{1}{2} n^{2/3 - \varepsilon} > n^{0.6}$. On the other hand, since Maker is trying to follow part $(1.2)$ of his strategy, it follows that there is no free edge $(w,v) \in E(K_n)$ such that $w$ is taken and $w' := f^{-1}(w)$ is open with respect to $T$. Hence, the number of open vertices with respect to $T$ at this point is at most $d_B(v) = o(n^{0.6})$, where this equality holds by part $(ii)$ of Lemma~\ref{maintainProperties}. This is clearly a contradiction. 

Next, we prove that Maker can embed $v', x'$, and $y'$ into appropriate available 
vertices, that is, that he can claim free edges 
$(u,x),(x,y)$, and $(y,v)$, where $x$ and $y$ are available vertices. 
Moreover, we prove that this entire phase takes at most $11 n^{\alpha}$ moves. 
Finally, we prove that Maker can then close $v$. It follows that Maker can follow this
part of his strategy without forfeiting the game. Let $A := V(K_n)
\setminus \left(f(S) \cup \{v\} \cup N_B(u) \cup N_B(v) \right)$.
Let $I \subseteq A$ be an independent set in Breaker's graph
of size $|I| \geq |A|/(\Delta(B) + 1) \geq 100n^{2\alpha}$. Such an
independent set exists since $|A| \geq \frac{1}{2} n^{2/3} - 1 -
o(n^{0.6})$ by parts $(i)$ and $(ii)$ of Lemma~\ref{maintainProperties},
and since Breaker's graph, induced on the vertices of $A$, has
maximum degree at most $o(n^{0.6})$ by property $(ii)$ of
Lemma~\ref{maintainProperties}. In the first $5n^{\alpha}$ moves of
this phase, Maker claims $5n^{\alpha}$ arbitrary edges of $\{(v,w) :
w \in I\}$. This is possible as $|\{(v,w) : w \in I\}| = |I| >
5n^{\alpha} + 5n^{2\alpha}$, and Breaker can claim at most
$5n^{2\alpha}$ of these edges. Let $N_v \subseteq I \cap N_M(v)$ be
an arbitrary set of size $5n^{\alpha}$. In his next $5n^{\alpha}$ moves, 
Maker claims $5n^{\alpha}$ edges of $\{(u,w) : w \in I \setminus N_v\}$. 
This is possible as $|\{(u,w) : w \in I \setminus N_v\}| = |I \setminus N_v| 
> 5n^{\alpha} + 10n^{2\alpha}$, and during this entire phase, Breaker 
has claimed at most $10n^{2\alpha}$ of these edges. Note that it is possible 
that Maker has already claimed some edges of $\{(v,w) : w \in I\} \cup 
\{(u,w) : w \in I \setminus N_v\}$ during some previous stage of the game. 
In this case he will need less than $10 n^{\alpha}$ moves to achieve his goal; 
clearly this does not harm him. Now, Maker can claim a free edge $(x,y)$, where 
$y \in N_v$ and $x \in (I \setminus N_v) \cap N_M(u)$. The required 
edge exists since there are at least $|N_v||(I \setminus N_v) \cap N_M(u)| 
\geq 25 n^{2\alpha}$ such edges in $K_n$, and during this entire phase, 
Breaker has claimed at most $10n^{2\alpha}$ of them.
Maker embeds $x'$ into $x$, $y'$ into $y$, and $v'$ into $v$. This completes 
the required path of length three. Note that embedding all these
vertices takes at most $11 n^{\alpha}$ moves. Finally, Maker closes
$v$ by embedding all of its neighbors (as in part $(1.1)$ of his strategy).

\textbf{A proof that Maker can follow Stage 2 of his strategy for Case I:}
 
Let $L'' \subseteq L'$ denote the set of leaves
which have not been embedded in Stage 1. Let $H = (X \cup Y, F)$ be 
the bipartite graph with $X = V(K_n) \setminus f(S)$, $Y = f(N_T(L''))$, 
and $F = \{(u,v) \in E(K_n \setminus B) : u \in X, v \in Y\}$. Note that, 
by the choice of $L'$, no two leaves of $L''$ have a common neighbor in $T$. 
Hence, in order to complete the embedding of $T$ in $K_n$, Maker has to 
claim the edges of a perfect matching of $H$, in a $(1 : b)$ game against 
Breaker, where $b \leq n^{\alpha}$. This is possible by Proposition~\ref{matchingE},
since $|Y| = |f(N_T(L''))| = |L''| = |X|$, $|L''| \geq \frac{1}{2}n^{\frac{2}{3}}$ 
by part $(i)$ of Lemma~\ref{maintainProperties}, $\Delta(B) \leq n^{0.6} \leq |X|^{0.95}$ 
in the beginning of Stage 2 by part $(ii)$ of Lemma~\ref{maintainProperties}, 
and $n^{\alpha} \leq \frac{|X|}{12\log_2 |X|}$. 

Note that Stage 1 and Stage 2 together last at most $(n + o(n^{0.6 + \alpha})) + 
O(n^{2/3} \log (n^{2/3})) = n + o(n)$ moves. 

This concludes the proof that Maker can win the $(1 : b)$ game ${\mathcal T}_n$
if $T$ is as in Case I.

\subsection{Following Maker's strategy for Case II}

In this subsection we prove that Maker can indeed follow every part of
his strategy for Case II. We prove this separately for Stage 1 and for Stage 2.
First, we prove the following lemma, which will be used in our proof that Maker can 
follow Stage 2 of his strategy for Case II.

\begin{lemma} \label{HamPathBetFixedab}
Let $\beta > 0$ and $0 < \gamma < 1/8$ be real numbers such that $\beta + 2 \gamma < 1$. 
Let $k$ be sufficiently large and let $q \leq k^{\gamma}$. Let $G=(V,E)$ be a
graph on $k$ vertices, with minimum degree at least $k - k^{\beta}$,
and let $a,b \in V$ be two vertices. Then, playing a $(1 : q)$ game
on $E$, Maker can build a Hamilton path, whose endpoints are $a$ and $b$, 
within $k + o(k)$ moves.
\end{lemma}

\begin{proof}
We present a strategy for Maker, and then prove that it is a winning
strategy. 

\textbf{Maker's strategy:} 
Maker's strategy is divided into two stages.

\textbf{Stage 1:} Let $\delta$ and $\delta'$ be real numbers 
such that $\delta' > \max \{\frac{1}{2} + 2 \gamma, \beta\}$ and 
$\delta' + 2 \gamma < \delta < 1$. Maker builds two vertex disjoint paths
$P_a = (a, v_1, v_2, \ldots, v_i)$ and $P_b = (b, u_1, u_2, \ldots, u_j)$, such that
$k^{\delta} \leq k - (|V(P_a)| + |V(P_b)|) \leq 2 k^{\delta}$. While building these two paths,
Maker ensures that $d_B(u) \leq 2 k^{\delta'}$ holds for every 
vertex $u \in (V \setminus (V(P_a) \cup V(P_b))) \cup ((End(P_a) \cup End(P_b)) 
\setminus \{a, b\})$.

At any point during this stage, a vertex $v \in (V \setminus (V(P_a) \cup V(P_b)))
\cup ((End(P_a) \cup End(P_b)) \setminus \{a, b\})$ is called \emph{dangerous} 
if $d_B(v) \geq k^{\delta'}$. Throughout this stage, Maker maintains a set 
$D \subseteq V$ of dangerous vertices and two paths of $G$ whose edges he has claimed, 
$P_a$ and $P_b$, where $a$ is an endpoint of $P_a$ and $b$ is an endpoint of $P_b$.
Initially, $D = \emptyset$, $P_a = (a)$, and $P_b = (b)$. Maker updates $D$ 
after each move (by either player). If at some
point, Maker is unable to follow the proposed strategy, then he
forfeits the game. Moreover, if after claiming $2k$ edges, Maker has
not yet won, then he forfeits the game. 

For as long as $k - (|V(P_a)| + |V(P_b)|) \geq 2 k^{\delta}$, Maker plays as follows:

\begin{description}
\item [$(1)$] If $D = \emptyset$, then Maker extends the shorter of 
his two paths (breaking ties arbitrarily). Assume without loss of generality
that currently $P_a$ is shorter than $P_b$. If $P_a = (a)$, then let $x = a$, 
otherwise, let $x$ denote the unique element of $End(P_a) \setminus \{a\}$. 
Maker extends $P_a$ by claiming a free edge $(x,w)$ for some $w \in V \setminus 
(V(P_a) \cup V(P_b))$. 

\item [$(2)$] If $D \neq \emptyset$, then Maker plays as follows. Let $v \in D$ 
be an arbitrary dangerous vertex. We distinguish between two sub cases: 
\begin{description}
\item [$(2.1)$] $v \in (End(P_a) \cup End(P_b)) \setminus \{a, b\}$. Maker claims a free edge $(v,w)$ for some $w \in V \setminus (V(P_a) \cup V(P_b))$. 

\item [$(2.2)$] $v \in V \setminus (V(P_a) \cup V(P_b))$.  
First, Maker adds $v$ to $P_a$. Let $x$ denote the unique element of 
$End(P_a) \setminus \{a\}$. Maker connects $v$ to $x$ via a path of length three, 
within at most $11 k^{\gamma}$ moves. Once $v \in V(P_a)$, Maker extends
$P_a$ by one more edge, as in $(1)$ above.
\end{description}
\end{description}

\textbf{Stage 2:} Let $G' = G[(V \setminus (V(P_a) \cup V(P_b))) \cup 
((End(P_a) \cup End(P_b)) \setminus \{a, b\})]$. Maker builds a
Hamilton connected subgraph of $G'$.

Note that, if Maker can indeed follow all parts of the proposed strategy, then, in
particular, he builds a Hamilton path in $G$ whose endpoints are $a$ and $b$ 
(though possibly not fast enough).

Next, we prove that Maker can indeed follow every part of
his suggested strategy. We prove this separately for Stage 1 and for Stage 2.
First, we prove the following two lemmas.

\begin{lemma} \label{dang1}
At any point during Stage 1 there are $o(\sqrt{k})$ dangerous
vertices.
\end{lemma}

\begin{proof}
By Maker's strategy, the game lasts at most $2k$ moves. Since, moreover, every
dangerous vertex has degree at least $k^{\delta'}$ in Breaker's graph, it follows that
there can be at most $2 \cdot \frac{2k^{1+\gamma}}{k^{\delta'}} = o(\sqrt{k})$ 
such vertices.
\end{proof}

\begin{lemma} \label{maintainProperties1}
The following two properties hold at any point during Stage 1:
\begin{description}
\item [$(i)$] $k - |V(P_a)| - |V(P_b)| \geq k^{\delta}$.
\item [$(ii)$] $d_B(v) \leq 2 k^{\delta'}$ holds for every $v \in (V \setminus (V(P_a) \cup V(P_b))) \cup ((End(P_a) \cup End(P_b)) \setminus \{a, b\})$.
\end{description}
\end{lemma}

\begin{proof}
\begin{description}
\item [$(i)$] According to his strategy, Maker tries to stop extending his paths at
the very moment $k - |V(P_a)| - |V(P_b)| \leq 2 k^{\delta}$ happens for the first time.
Maker can stop at this exact moment unless he is in the middle of part $(2.2)$ of his 
strategy. In this case he adds at most $3$ additional vertices. It follows that
$k - |V(P_a)| - |V(P_b)| \geq 2 k^{\delta} - 3 \geq k^{\delta}$ as claimed.

\item [$(ii)$] This holds trivially for any non-dangerous vertex. By Maker's
strategy, adding any dangerous vertex to the interior of $P_a \cup P_b$, 
requires at most $11 k^{\gamma}$ moves. By Lemma~\ref{dang1} there are
$o(\sqrt{k})$ dangerous vertices at any point of the game. Moreover,
as long as $D \neq \emptyset$, all of Maker's moves are dedicated to
adding dangerous vertices to the interior of $P_a \cup P_b$. Hence, when 
Maker tries to add a dangerous vertex $v$ to the interior of $P_a \cup P_b$, 
it holds that $d_B(v) \leq k^{\delta'} + o(k^{\gamma}k^{\gamma} \sqrt{k}) 
\leq 2 k^{\delta'}$. Since, unless he has already won, Maker continues adding 
dangerous vertices to the interior of $P_a \cup P_b$, for as long as $D \neq \emptyset$,
it follows that $d_B(w) \leq 2 k^{\delta'}$ holds for every $w \in (V \setminus (V(P_a) \cup V(P_b))) \cup ((End(P_a) \cup End(P_b)) \setminus \{a, b\})$ as claimed.
\end{description}
\end{proof}

\textbf{Stage 1:} 
If $D = \emptyset$, then Maker can extend either path. Indeed, assume that Maker wishes
to extend $P_a$. Let $x$ denote the unique element of $End(P_a) \setminus \{a\}$ (or $x=a$ 
if $P_a = (a)$). It follows by part $(i)$ of Lemma~\ref{maintainProperties1} that 
$k - |V(P_a)| - |V(P_b)| \geq k^{\delta}$. Since $d_B(x) \leq k^{\delta'}$, the minimum
degree of $G$ is at least $k - k^{\beta}$, and $\delta > \delta' > \beta$, it follows that there exists a vertex $w \in V \setminus (V(P_a) \cup V(P_b))$ for which the edge $(x,w)$ is free. 

If $D \neq \emptyset$, then Maker has to add some vertex $v \in D$ to the interior of $P_a \cup P_b$. The fact that he can indeed achieve this goal, and moreover do so within at most $11 k^{\gamma}$ moves, follows by essentially the same argument used to show that Maker can follow part $(1)$ of Stage 1 of his strategy for Case I. We omit the straightforward details (note that here is where we use the fact that $\delta > \delta' + 2 \gamma$). 

\textbf{Stage2:}
It follows by the fact that Maker can follow his strategy for Stage 1 without forfeiting the game and by part $(i)$ of Lemma~\ref{maintainProperties1} that $|V(G')| = \Theta(k^{\delta})$. It follows by part $(ii)$ of Lemma~\ref{maintainProperties1} that the minimum degree of $G'$ is at least $|V(G')| - 1 - 2 k^{\delta'} = |V(G')| - \Theta(|V(G')|^{\delta'/\delta})$. Since, moreover, $q \leq k^{\gamma} \leq \frac{|V(G')|}{\log^2 |V(G')|}$, it follows by Proposition~\ref{HamiltonConnected} that Maker can indeed build a Hamilton connected subgraph of $G'$. Furthermore, he can do so within $O(|V(G')| \log^2 |V(G')|) = o(k)$ moves.

\end{proof}

\textbf{A proof that Maker can follow Stage 1 of his strategy for Case II:}

Clearly, there are at most $|D_1| + |D_{>2}|$ inclusion maximal paths with at least one endpoint in $D_{>2}$. Since, except for vertices of $D_{>2}$, Maker embeds only vertices of bare paths whose length is at most $n^{0.2}$, it follows that, during Stage 1 Maker tries to embed at most $|D_{>2}| + (|D_1| + |D_{>2}|) n^{0.2} \leq 3 |D_1| n^{0.2} \leq 3 n^{2/3 + \varepsilon + 0.2} = o(n^{0.92})$ vertices, where the first inequality above follows by Lemma~\ref{leaves}. Since, moreover, $b \leq n^{\alpha}$ where $\alpha < 0.005$, it follows that at any point during Stage 1, the maximum degree in Breaker's graph is $o(n^{0.95})$, whereas the number of vertices that have not yet been embedded is $(1-o(1))n$. Hence, there are always free edges that extend the embedding of $F$. 

\textbf{A proof that Maker can follow Stage 2 of his strategy for Case II:}

By Maker's strategy, it remains to embed $\ell$ bare paths, for some $1 \leq \ell \leq 2 n^{2/3 + \varepsilon}$. For every $1 \leq i \leq \ell$, let $n_i$ denote the length of the $i$th path and let $a_i$ and $b_i$ denote its endpoints. Note that $n_i \geq n^{0.2}$ and that $a_i$ and $b_i$ were already embedded in Stage 1. Let $\tilde{F}$ denote the set of vertices of $K_n$ into which the vertices of $F$ were embedded in Stage 1. First, Maker partitions the remaining board into $\ell$ disjoint parts $V(K_n) \setminus \tilde{F} = V_1 \cup \ldots \cup V_{\ell}$ such that
\begin{description}
\item [$(i)$] $|V_i| = n_i - 1$.
\item [$(ii)$] The minimum degree of $G_i := (K_n \setminus (M \cup B))[V_i \cup \{a_i, b_i\}]$ is at least $n_i - 
10 n_i n^{-0.05}$.
\end{description}

Such a partition exists by Lemma~\ref{randomPartition}. Maker plays a $(1 : b)$ game
on $\bigcup_{i=1}^{\ell} G_i$, which he wins if and only if, for every $1 \leq i \leq \ell$, he is able to build a Hamilton path of $G_i$ between $a_i$ and $b_i$. Clearly, if Maker wins this game, then he completes the embedding of $T$ in $K_n$. By Theorem~\ref{ParallelGames}, in order to prove that Maker wins the aforementioned game, it suffices to prove that for every $1 \leq i \leq \ell$, Maker can win a $\left(1 : b \left(1 + \log \left(\ell + \lceil \frac{n^2}{b+1} \rceil \right)\right)\right)$ game played on $E(G_i)$, which he wins if and only if he is able to build a Hamilton path of $G_i$ whose endpoints are $a_i$ and $b_i$. However, since $|V(G_i)| \geq n^{0.2}$ and $b \left(1 + \log \left(\ell + \lceil \frac{n^2}{b+1} \rceil \right)\right) \leq 2n^{\alpha} \log n \leq |V(G_i)|^{1/40}$, this follows by Property $(ii)$ above and by Lemma~\ref{HamPathBetFixedab}. Moreover, since, by Lemma~\ref{HamPathBetFixedab}, for every $1 \leq i \leq \ell$, Maker wins the game on $E(G_i)$ in $n_i + o(n_i)$ moves, it follows by Theorem~\ref{ParallelGames} that Maker wins the game on $\bigcup_{i=1}^{\ell} G_i$ in $\sum_{i=1}^{\ell} (n_i + o(n_i))$ moves.

Note that Stage 1 and Stage 2 together last at most
$o(n) + (\sum_{i=1}^{\ell} (n_i + o(n_i))) = n + o(n)$ moves. 

This concludes the proof that Maker can win the $(1 : b)$ game ${\mathcal T}_n$
if $T$ is as in Case II.

\section{Concluding remarks and open problems} \label{sec::openprob}

\begin{description}
\item [Breaker's bias and the maximum degree of $T$.]
In this paper it is proved that, given any tree $T$ on $n$ vertices whose maximum degree is not too large, Maker can build a copy of $T$ within $n + o(n)$ moves, when playing a biased game on $E(K_n)$. While the obtained upper bound on the duration of the game is clearly very close to being optimal for sufficiently large $n$, the upper bounds on the maximum degree of $T$ and on Breaker's bias are probably quite far from being best possible. It would be interesting to improve either bound, even at the expense of the other. In particular we offer the following two questions:
\begin{problem} \label{degree}
Let $n$ be sufficiently large and let $T=(V,E)$ be a tree on $n$ vertices. What is the largest integer $d = d(n)$ for which Maker has a winning strategy for the $(1 : 1)$ game ${\mathcal T}_n$, assuming that $\Delta(T) \leq d$?
\end{problem}

\begin{problem} \label{bias}
Let $n$ be sufficiently large and let $T=(V,E)$ be a tree on $n$ vertices with constant maximum degree. What is the largest integer $b = b(n)$ for which Maker has a winning strategy for the $(1 : b)$ game ${\mathcal T}_n$?
\end{problem}   

\item [Other spanning graphs.]
It would be interesting to analyze analogous games for general spanning graphs (not necessarily trees) of bounded degree. It was proved in~\cite{HS} that, for sufficiently large $n$, Maker can build a Hamilton cycle of $K_n$ in a $(1 : 1)$ game, within $n+1$ moves. However, even if we restrict our attention to $(1 : 1)$ games on $E(K_n)$ in which Maker's goal is to build a copy of some predetermined $2$-regular spanning graph, we cannot expect him to win very quickly. Indeed, for every positive integer $n$ such that $3 \mid n$, let ${\mathcal TF}_n$ denote the game whose board is $E(K_n)$ and whose winning sets are all triangle factors of $K_n$. The following result was observed by Tibor Szab\'o and the third author.

\begin{theorem} \label{3-factor}
Maker cannot win the $(1 : 1)$ game ${\mathcal TF}_n$ in less than $7n/6$ moves.
\end{theorem}

\begin{proof}
We describe Breaker's strategy and then show that indeed, following this
strategy delays Maker's win for at least $7n/6$ moves. Without loss of
generality we assume that Maker is the first player (otherwise, Breaker claims
an arbitrary edge in his first move). 

\textbf{Breaker's strategy:}
For every $i \geq 1$, let $e_i = (x_i, y_i)$ denote the edge claimed by 
Maker in his $i$th move. In his $i$th move, Breaker plays as follows.
If there exists a vertex $z \in V(K_n)$ such that $(x_i, z)$ was previously
claimed by Maker and $(y_i, z)$ is free, then Breaker claims $(y_i, z)$ 
(if there are several such edges $(y_i, z)$, then Breaker claims one of them arbitrarily).
If there exists a vertex $z \in V(K_n)$ such that $(y_i, z)$ was previously
claimed by Maker and $(x_i, z)$ is free, then Breaker claims $(x_i, z)$
(if there are several such edges $(x_i, z)$, then Breaker claims one of them arbitrarily). 
Otherwise, Breaker claims an arbitrary edge.

\begin{claim} \label{Claim1}
If Breaker follows the proposed strategy, then there is at least one vertex
of degree at least 3 in any triangle in Maker's graph. 
\end{claim}

\begin{proof}
Let $(x,y), (y,z)$ and $(x,z)$ be three edges in Maker's graph. Assume that 
the first of these three edges to be claimed by Maker was $(x,y)$ and the second
was $(y,z)$. Let $i$ denote the number of the move in which Maker has claimed $(y,z)$.
Since, on his $i$th move, Breaker did not claim $(x,z)$, it follows by the description
of his strategy that he must have claimed an edge $(y,u)$ or $(z,u)$ for some 
vertex $u \neq x$. Hence, $y$ or $z$ will have degree at least 3 in Maker's
graph after he will claim $(x,z)$. 
\end{proof}

Consider Maker's graph $M$ immediately after it admits a triangle factor 
for the first time. Clearly the minimum degree of $M$ is at least 2. 
Let $F$ be an arbitrary triangle factor of $M$. It follows by Claim~\ref{Claim1} 
that there is at least one vertex of degree at least three in every triangle of 
$F$. Hence, $2 e(M) = \sum_{v \in V} d_M(v) \geq 3 \frac{n}{3} + 2 \frac{2n}{3} = 
\frac{7n}{3}$. It follows that building $M$ took Maker at least $7n/6$ moves.  
\end{proof}

\end{description}

\end{document}